\documentclass[a4paper, 12pt, twoside]{article}
\usepackage{amsmath, amsthm, amsfonts, amssymb, graphicx, color}
\usepackage{esint,comment}
\numberwithin{equation}{section}
     \newtheorem{thm}{Theorem}[section]
     \newtheorem{cor}[thm]{Corollary}
     \newtheorem{prop}[thm]{Proposition}
     \newtheorem{lem}[thm]{Lemma}
\theoremstyle{definition}
      \newtheorem{defn}{Definition}[section]
     
\theoremstyle{remark}
     \newtheorem{rem}{Remark}[section]

\newcommand{\N}{\mathbb{N}}
\newcommand{\R}{\mathbb{R}}
\newcommand{\Q}{\mathbb{Q}}
\newcommand{\T}{\mathbb{T}}
\newcommand{\Z}{\mathbb{Z}}

\newcommand{\cB}{\mathcal{B}}

\newcommand{\cL}{\mathcal{L}}

\newcommand{\cP}{\mathcal{P}}

\newcommand{\cS}{\mathcal{S}}

\newcommand{\CMO}{\mathrm{CMO}}

\newcommand{\M}{\mathrm{M}}

\newcommand{\supp}{\operatorname{supp}}

\newcommand{\loc}{\mathrm{loc}}
\newcommand{\comp}{\mathrm{comp}}

\newcommand{\ls}{\lesssim}
\newcommand{\gs}{\gtrsim}

\newcommand{\cGdec}{\mathcal{G}^{\rm dec}}

\newcommand{\Ci}{C^{\infty}}

\newcommand{\Cic}{C^{\infty}_{\comp}}

\newcommand{\vp}{\varphi}

\newcommand{\dlim}{\displaystyle\lim}
\newcommand{\dlimsup}{\displaystyle\limsup}

\newcommand{\dsup}{\displaystyle\sup}

\begin{document}

\baselineskip=18pt

\title{The distance in Morrey spaces to $C^{\infty}_{\mathrm{comp}}$}
%\title{Predual spaces of the spaces generated by blocks}
%[Bi-predual spaces of Campanato spaces]

\author{Satoshi Yamaguchi}

\date{}
%\footnotetext{$^*$Corresponding author. }

\maketitle

\begin{center}
\small\it
Department of Mathematics, Ibaraki University, Mito, Ibaraki 310-8512, Japan \\[1ex]
Email:
23nd203a@vc.ibaraki.ac.jp;
satoshiyamaguchi.1998@gmail.com
\end{center}

\begin{abstract}
In this paper we characterize the distance between the function $f$
and the set $\Cic(\R^d)$ 
in generalized Morrey spaces $L_{p,\phi}(\R^d)$ with variable growth condition. 
We also prove that the bi-dual of $\overline{\Cic(\R^d)}^{L_{p,\phi}(\R^d)}$ 
is $L_{p,\phi}(\R^d)$.  
As an application of the characterization of the distance
we show the boundedness of Calder\'{o}n-Zygmund operators on $\overline{\Cic(\R^d)}^{L_{p,\phi}(\R^d)}$.
By the duality we also see that 
these operators are bounded on its dual and bi-dual spaces. 
%%%---------------------------------------------------------
\footnote[0]
{2020 $\mathit{Mathematics\ Subject\ Classification}.$
\ 42B35, 46E30, 46B10, 42B20 

\ \ $\mathit{Key\ words\ and\ phrases.}$
\ Morrey space, block space, variable growth condition, predual, singular integral.}
\end{abstract}

%%%%%%%%%%%%%%%%%%%%%%%%%%%%%%%%%%%%%%%%%%%%%%%%%%%%
\section{Introduction}\label{s:intro}
%%%%%%%%%%%%%%%%%%%%%%%%%%%%%%%%%%%%%%%%%%%%%%%%%%%%
In Euclidean space $\R^d$ 
let $B=B(x,r)$ denote the open ball centered at $x$ and of radius $r$.
For a function $f\in L^1_{\loc}(\R^d)$ and a ball $B$, let
\begin{equation}
 \M_p(f,B)=\left(\frac{1}{|B|}\int_B|f(y)|^p\,dy\right)^{1/p},
 \quad
 p\in[1,\infty),
\end{equation}
where $|B|$ is the Lebesgue mesure of $B$.
Then the Morrey space $L_{p,\lambda}(\R^d)$ is defined to be the set of all functions $f\in L^1_{\loc}(\R^d)$
such that
\begin{equation}
 \|f\|_{L_{p,\lambda}}=\sup_{B=B(x,r)}\frac{\M_p(f,B)}{r^{\lambda}}<\infty,
 \quad
 p\in[1,\infty), \ \lambda\in[-d/p,0], 
\end{equation}
where the supremum is taken over all balls $B$ in $\R^d$. 
If $\lambda=-d/p$, then $L_{p,\lambda}(\R^d)=L^p(\R^d)$.
If $\lambda=0$, then $L_{p,\lambda}(\R^d)=L^{\infty}(\R^d)$ by the Lebesgue differentiation theorem.

It is known that $\Cic(\R^d)$,
the set of all infinitely differentiable functions with compact support, 
is not dense in $L_{p,\lambda}(\R^d)$ 
for $p\in[1,\infty)$ and $\lambda\in(-d/p,0)$,
while it is dense in $L^p(\R^d)$. 
For example, let 
\begin{equation*}
 f_{x_0}(x)=|x-x_0|^{\lambda/p}\chi_{B(x_0,r_0)},
\end{equation*}
where $\chi_E$ is the characteristic function of the mesurable set $E$ in $\R^d$.
Then $f_{x_0}$ is in $L_{p,\lambda}(\R^d)$, but 
\begin{equation*}
 \|f_{x_0}-h\|_{L_{p,\lambda}}\ge2^{-p-1}|S^{d-1}|
\end{equation*}
for any continuous function $h$, where $S^{d-1}=\{x\in\R^d:|x|=1\}$.
This example was given by Zorko~\cite{Zorko1986}.

In this paper 
we consider generalized Morrey space $L_{p,\phi}(\R^d)$ for 
$p\in[1,\infty)$ and $\phi:\R^d\times(0,\infty)\to(0,\infty)$.
For a ball $B=B(x,r)$ we write $\phi(B)=\phi(x,r)$.
Then $L_{p,\phi}(\R^d)$ is defined to be the set of all functions $f$ such that
%the following functional is finite: 
\begin{equation*}
 \|f\|_{L_{p,\phi}}=\sup_{B}\frac{\M_p(f,B)}{\phi(B)}<\infty,
\end{equation*}
where
the supremum is taken over all balls $B$ in $\R^d$.
The purpose of this paper is to characterize the distance 
between the function $f$ and the set $\Cic(\R^d)$ in $L_{p,\phi}(\R^d)$,
which is defined by
\begin{equation*}
 d(f,\Cic(\R^d))=\inf_{g\in\Cic(\R^d)}\|f-g\|_{L_{p,\phi}}.
\end{equation*}

As a corollary we get the characterization 
on $\overline{\Cic(\R^d)}^{L_{p,\phi}(\R^d)}$
like the Uchiyama's characterization on $\CMO(\R^d)$ 
in \cite{Uchiyama1978}, 
where $\overline{\Cic(\R^d)}^{L_{p,\phi}(\R^d)}$ 
is the closure of $\Cic(\R^d)$ with respect to $L_{p,\phi}(\R^d)$.
For the case $\phi:(0,\infty)\to(0,\infty)$,
the space $\overline{\Cic(\R^d)}^{L_{p,\phi}(\R^d)}$
was introduced by Zorko~\cite{Zorko1986}
as a subspace of $L_{p,\phi}(\R^d)$ with nice properties.

We also prove that the bi-dual space of 
$\overline{\Cic(\R^d)}^{L_{p,\phi}(\R^d)}$ is $L_{p,\phi}(\R^d)$. 
It is known that
the block space $B^{[\phi,p']}(\R^d)$ is one of predual spaces of $L_{p,\phi}(\R^d)$
by Nakai~\cite{Nakai2008AMS}, where $1/p+1/p'=1$.
The block space (space generated by blocks) was introduced 
by Lu, et al.~\cite{Lu-Taibleson-Weiss1982} and Taibleson and Weiss~\cite{Taibleson-Weiss1983}.
Nakai's result is a generalization of the result by Long~\cite{Long1984}. 
We show that the dual space of $\overline{\Cic(\R^d)}^{L_{p,\phi}(\R^d)}$ is $B^{[\phi,p']}(\R^d)$.

%It is known that the dual space of $B^{[\phi,p']}(\R^d)$ is $L_{p,\phi}(\R^d)$ by \cite{Nakai2008AMS}, 
%where $B^{[\phi,p']}(\R^d)$ is the block space and $1/p+1/{p'}=1$.
%We show that the dual space of $\overline{\Cic(\R^d)}^{L_{p,\phi}(\R^d)}$ is $B^{[\phi,p']}(\R^d)$.
%\begin{equation}
% \Big(\overline{\Cic(\R^d)}^{L_{p,\phi}(\R^d)}\Big)^*=B^{[\phi,p']}(\R^d).
%\end{equation}
%$(\overline{\Cic(\R^d)}^{L_{p,\phi}(\R^d)})^*=B^{[\phi,p']}(\R^d)$. 

As an application of the characterization of the distance $d(f,\Cic(\R^d))$ 
we show the boundedness of Calder\'{o}n-Zygmund operators on $\overline{\Cic(\R^d)}^{L_{p,\phi}(\R^d)}$. 
This result is an extension of the result in 
\cite{Rosenthal-Triebel2014}. 
Moreover, the proof method differs from that in 
\cite{Rosenthal-Triebel2014}. 
Through the dual and bi-dual operators we also see that 
these operators are bounded on $B^{[\phi,p']}(\R^d)$ and on $L_{p,\phi}(\R^d)$. 

%%%-------------------------------------------------------------------
In this pape we consider the following conditions on $\phi$.
We say that a function $\phi:\R^d\times(0,\infty)\to(0,\infty)$ 
satisfies the doubling condition (resp. nearness condition) if
there exists a positive constant $C$ such that,
for all $x,y\in\R^d$ and $r,s\in(0,\infty)$,
\begin{align}\label{DC}\tag{DC}
 \frac1C\le\frac{\phi(x,r)}{\phi(x,s)}&\le C,
 \quad\text{if} \ \ \frac12\le\frac{r}{s}\le2  \\
 \bigg(\text{resp.}\ 
 \frac1C\le\frac{\phi(x,r)}{\phi(y,r)}&\le C,
 \quad\text{if} \ \ |x-y|\le r\bigg).                    \label{NC}\tag{NC}
\end{align}
We say that $\phi$ is almost increasing (resp. almost decreasing) if
there exists a positive constant $C$ such that, for all $x\in\R^d$ and $r,s\in(0,\infty)$,
\begin{align}
 &\phi(x,r)\le C\phi(x,s),\quad\text{if $r<s$}  \label{AI}\tag{AI}
\\
 (\text{resp.}\ &C\phi(x,r)\ge \phi(x,s), \quad\text{if $r<s$}). \label{AD}\tag{AD}
\end{align}
For two functions $\phi,\psi:\R^d\times(0,\infty)\to(0,\infty)$,
we write $\phi\sim\psi$ if
there exists a positive constant $C$ such that,
for all $x\in\R^d$ and $r\in(0,\infty)$,
\begin{equation}\label{rho kappa}
 \frac1C\le\frac{\phi(x,r)}{\psi(x,r)}\le C.
\end{equation}
We also define 
$\cGdec_p$ as the set of all functions $\phi:\R^d\times(0,\infty)\to(0,\infty)$
such that 
$r\mapsto r^{d/p}\phi(x,r)$ is almost increasing 
and that
$r\mapsto\phi(x,r)$ is almost decreasing.
That is,
there exists a positive constant $C$ such that, 
for all $x\in\R^d$ and $r,s\in(0,\infty)$,
\begin{equation*}
 r^{d/p}\phi(x,r)\le Cs^{d/p}\phi(x,s),
 \quad
 C\phi(x,r)\ge \phi(x,s), 
 \quad
 \text{if} \ r<s.
\end{equation*}

%%%-------------------------------------------------------------------
It is known that, if $\phi$ is in $\cGdec_p$ and satisfies \eqref{NC},
then $\chi_B\in L_{p,\phi}(\R^d)$ for each ball $B$, see \cite{Tang-Nakai-Yang-Zhu-preparation}.
Consequently, $\Cic(\R^d)\subset L_{p,\phi}(\R^d)$.
We give its proof in Section~\ref{s:duality} for convenience.

%%%-------------------------------------------------------------------
For example, let
\begin{equation}\label{phi lambda}
  \phi(x,r)=
  \begin{cases}
   r^{\lambda(x)}, & 0<r<1, \\
   r^{\lambda_*}, & 1\le r<\infty,
  \end{cases}
\end{equation}
where $\lambda(\cdot):\R^d\to(-\infty,\infty)$ and $\lambda_*\in(-\infty,\infty)$.
Let 
\begin{equation*}
 \lambda_+=\sup_{x\in\R^d}\lambda(x), \quad \lambda_-=\inf_{x\in\R^d}\lambda(x).
\end{equation*}
If $-d/p\le\lambda_-\le\lambda_+\le0$ and $-d/p\le\lambda_*\le0$,
then
$\phi$ is in $\cGdec_p$.
Moreover, if $\lambda(\cdot)$ is log-H\"older continuous,
that is,
there exists a positive constant $C$ such that, for all $x,y\in\R^d$,
\begin{equation*} 
  |\lambda(x)-\lambda(y)|\le \frac{C}{\log(e/|x-y|)}
  \quad\text{if}\quad
  0<|x-y|<1,
\end{equation*}
then $\phi$ satisfies \eqref{NC}, see \cite[Proposition~3.3]{Nakai2010RMC}.

To state the main result we define $A_{p,\phi}$ for $p\in[1,\infty)$ and $\phi:\R^d\times(0,\infty)\to(0,\infty)$
as follows:
\begin{align*}
 A_{p,\phi}(f)
 &=\max\Bigg\{
  \dlimsup_{r\to+0}\sup_{x\in\R^d}\frac{\M_p(f,B(x,r))}{\phi(x,r)}, \\
 &\quad
 \dlimsup_{r\to\infty}\sup_{x\in\R^d}\frac{\M_p(f,B(x,r))}{\phi(x,r)}, \quad
  \dsup_{r>0}\dlimsup_{|x|\to\infty}\frac{\M_p(f,B(x,r))}{\phi(x,r)}
  \Bigg\}. 
\end{align*}
Then the main result is the following:

%%%==============================================
\begin{thm}\label{thm:distance}
%%%==============================================
Let $p\in[1,\infty)$, 
and let $\phi$ be in $\cGdec_p$ and satisfy \eqref{NC}.
Then there exists a positive constant $C$ such that,
for all $f\in L_{p,\phi}(\R^d)$,
\begin{equation}\label{d le A}
 d(f,\Cic(\R^d))\le CA_{p,\phi}(f),
\end{equation}
Moreover, if 
\begin{align}
 \lim_{r\to+0}\inf_{x\in B(0,M)}\phi(x,r)&=\infty
 \quad 
 \text{for each $M>0$},                                         \label{lim1} \\
 \lim_{r\to\infty}r^{d/p}\phi(0,r)&=\infty,       \label{lim2}
\end{align}
then, 
for all $f\in L_{p,\phi}(\R^d)$,
\begin{equation}
 A_{p,\phi}(f)\le d(f,\Cic(\R^d))\le CA_{p,\phi}(f).
\end{equation}
\end{thm}

%=================================================
\begin{rem}\label{rem:distance}
%=================================================
The inequality \eqref{d le A} is useful to prove 
the boundedness of Calder\'{o}n--Zygmund operators in Section~\ref{s:CZO}. 
\end{rem}

%=================================================
\begin{cor}\label{cor:closure}
%=================================================
Let $p\in[1,\infty)$, and let $\phi$ be in $\cGdec_p$ and satisfy \eqref{NC}, 
\eqref{lim1} and \eqref{lim2}. 
Let $f\in L_{p,\phi}(\R^d)$.
Then $f\in\overline{\Cic(\R^d)}^{L_{p,\phi}(\R^d)}$
if and only if $f$ satisfies the following three conditions:
\begin{enumerate}
\item \label{rto0}
$\dlim_{r\to+0}\sup_{x\in\R^d}\frac{\M_p(f,B(x,r))}{\phi(x,r)}=0$,
\item \label{rtoinfty}
$\dlim_{r\to\infty}\sup_{x\in\R^d}\frac{\M_p(f,B(x,r))}{\phi(x,r)}=0$,
\item \label{xtoinfty}
$\dsup_{r>0}\dlim_{|x|\to\infty}\frac{\M_p(f,B(x,r))}{\phi(x,r)}=0$.
\end{enumerate}
\end{cor}

%=================================================
\begin{rem}\label{rem:closure}
%=================================================
We do not need \eqref{lim1} and \eqref{lim2} to prove that,
if $f$ satisfies (i)--(iii), then $f\in\overline{\Cic(\R^d)}^{L_{p,\phi}(\R^d)}$.
Conversely, if $\phi(x,r)\equiv1$ which does not satisfy \eqref{lim1},
then $L_{p,\phi}(\R^d)=L^\infty(\R^d)$ and 
(i) of Corollary~\ref{cor:closure} fails for $f\in\Cic(\R)$ 
such that $f(x)=1$ near the origin, 
since $\big(\fint_{-r}^{r}|f|^p\big)^{1/p}=1$ as $r\to+0$.
Also, if $\phi(x,r)=r^{-d/p}$ which does not satisfy \eqref{lim2},
then $L_{p,\phi}(\R^d)=L^p(\R^d)$  and 
(ii) of Corollary~\ref{cor:closure} fails for 
$f\ne0$ in $L^p(\R)=\overline{\Cic(\R)}^{L^p(\R)}$, 
since $\big(\int_{-r}^{r}|f|^p\big)^{1/p}\to\|f\|_{L^p}$ as $r\to\infty$.
\end{rem}

%=================================================
\begin{rem}\label{rem:torus}
%=================================================
On the torus $\T^d$ instead of $\R^d$, 
from the proof of Corollary~\ref{cor:closure} 
we see that 
\begin{equation*}
 f\in\overline{\Ci(\T^d)}^{L_{p,\phi}(\T^d)}
 \quad\text{if and only if}\quad 
 \lim_{r\to+0}\sup_{x\in\T^d}\frac{\M_p(f,B(x,r))}{\phi(x,r)}=0. 
\end{equation*}
This result is an extension of 
\cite[Theorem~1.1 (1), (4)]{Izumi-Sato-Yabuta2014}. 
\end{rem}

%=================================================
\begin{rem}\label{rem:campanato}
%=================================================
If we assume the condition \eqref{CZ-C} below, 
%$\dint_r^\infty \dfrac{\phi(x,t)}{t}\,dt \ls \phi(x,r)$ 
%for all $x\in\R^d$ and $r\in(0,\infty)$, 
then $L_{p,\phi}(\R^d)$ is equivalent to 
the Campanato space $\cL_{p,\phi}(\R^d)$ modulo constant functions, 
and then Corollary~\ref{cor:closure} follows from 
Theorem~3.2 in \cite{Yamaguchi-Nakai-Shimomura2023AMS} 
as a corollary. 
However we do not assume \eqref{CZ-C} in the main result. 
\end{rem}

The following is an immediate consequence of Corollary~\ref{cor:closure}. 

%=================================================
\begin{cor}\label{cor:lambda}
%=================================================
Let $p\in[1,\infty)$. 
Let $\phi:\R^d\times(0,\infty)\to(0,\infty)$ be defined by \eqref{phi lambda}.
Assume that 
$-d/p\le\lambda_-\le\lambda_+<0$, $-d/p<\lambda_*\le0$
and that $\lambda(\cdot)$ is log-H\"older continuous.
Let $f\in L_{p,\phi}(\R^d)$.
Then $f\in\overline{\Cic(\R^d)}^{L_{p,\phi}(\R^d)}$
if and only if $f$ satisfies the following three conditions:
\begin{enumerate}
\item 
$\dlim_{r\to+0}\sup_{x\in\R^d}\frac{\M_p(f,B(x,r))}{r^{\lambda(x)}}=0$.
\item 
$\dlim_{r\to\infty}\sup_{x\in\R^d}\frac{\M_p(f,B(x,r))}{r^{\lambda_*}}=0$.
\item 
$\dlim_{|x|\to\infty}\M_p(f,B(x,r))=0$ \ for each $r>0$.
\end{enumerate}
\end{cor}

In the above theorem the condition \ref{xtoinfty} 
follows from that 
\begin{equation*}
 \lim_{|x|\to\infty}\frac{\M_p(f,B(x,r))}{\phi(x,r)}=0 \ \text{for each $r>0$},
\end{equation*}
and that
$r^{\lambda_+}\le\phi(x,r)\le r^{\lambda_-}$ if $0<r<1$, or, 
$\phi(x,r)=r^{\lambda_*}$ if $r\ge1$. 

We also prove the following theorem:

%=================================================
\begin{thm}\label{thm:bidual}
%=================================================
Let $p,q\in(1,\infty)$ and $1/p+1/q=1$, 
and let $\phi$ be in $\cGdec_p$ and satisfy \eqref{NC}, \eqref{lim1} and \eqref{lim2}. 
Then the bi-dual space of $\overline{\Cic(\R^d)}^{L_{p,\phi}(\R^d)}$ is $L_{p,\phi}(\R^d)$, that is,
\begin{equation}\label{CgM-gM}
 \Big(\overline{\Cic(\R^d)}^{L_{p,\phi}(\R^d)}\Big)^{**}=L_{p,\phi}(\R^d).
\end{equation}
\end{thm}

It is known that
the block space $B^{[\phi,p']}(\R^d)$ is one of predual spaces of $L_{p,\phi}(\R^d)$
by \cite{Nakai2008AMS}, where $1/p+1/p'=1$.
%The block space (space generated by blocks) was introduced 
%by Lu, et al.~\cite{Lu-Taibleson-Weiss1982} and Taibleson and Weiss~\cite{Taibleson-Weiss1983}.
%Nakai's result is a generalization of the result by Long~\cite{Long1984}. 
To prove Theorem~\ref{thm:bidual} we will show that
\begin{equation}\label{dual Intro}
 \Big(\overline{\Cic(\R^d)}^{L_{p,\phi}(\R^d)}\Big)^*=B^{[\phi,p']}(\R^d),
\end{equation}
by using the method by Coifman and Weiss~\cite{Coifman-Weiss1977}. 
The duality \eqref{dual Intro} is a generalization of the results by Izumi, et al. \cite{Izumi-Sato-Yabuta2014} and 
by Rosenthal and Triebel~\cite{Rosenthal-Triebel2015},
that treated $L_{p,\lambda}(\T^d)$ and $L_{p,\lambda}(\R^d)$, respectively.

The organization of this paper is as follows. 
We first investigate the mollifier for the functions in $L_{p,\phi}(\R^d)$ in Section~\ref{s:mollifier}. 
Then, using the properties of the mollifier, we show  Theorem~\ref{thm:distance} and the duality \eqref{dual Intro} 
in Sections~\ref{s:distance} and \ref{s:duality}, respectively.
As an application of the characterization of the distance we prove the boundedness of 
Calder\'{o}n-Zygmund operators 
on $\overline{\Cic(\R^d)}^{L_{p,\phi}(\R^d)}$ 
in Section~\ref{s:CZO}. 
To prove the theorems we borrow ideas from the proofs in 
\cite{Yamaguchi-Nakai-Shimomura2023AMS} 
which dealt with Campanato spaces modulo constant functions. 
We can use weaker assumptions than 
\cite{Yamaguchi-Nakai-Shimomura2023AMS}. 
More precisely, we use the blocks instead of atoms, 
the mean value instead of the mean oscillation, 
and the cutoff by the characteristic functions of balls 
instead of continuous functions.
%Actually, some parts are simpler than \cite{Yamaguchi-Nakai-Shimomura2023AMS}. 
Moreover, to define the integral operators 
on $L_{p,\phi}(\R^d)$ %or $\overline{\Cic(\R^d)}^{L_{p,\phi}(\R^d)}$, 
we do not need the cancellation property on their kernels,
differ from the Campanato spaces modulo constant functions.

At the end of this section, we make some conventions. 
Throughout this paper, we always use $C$ to denote a positive constant 
that is independent of the main parameters involved 
but whose value may differ from line to line.
Constants with subscripts, such as $C_p$, 
are dependent on the subscripts.
If $f\le Cg$, we then write $f\ls g$ or $g\gs f$; 
and if $f \ls g\ls f$, we then write $f\sim g$.

%%%%%%%%%%%%%%%%%%%%%%%%%%%%%%%%%%%%%%%%%%%%%%%%%%%%%%%%%%%%%%%%%%%%%%
%%%%%%%%%%%%%%%%%%%%%%%%%%%%%%%%%%%%%%%%%%%%%%%%%%%%%%%%%%%%%%%%%%%%%%
\section{Mollifier}\label{s:mollifier}
%%%%%%%%%%%%%%%%%%%%%%%%%%%%%%%%%%%%%%%%%%%%%%%%%%%%%%%%%%%%%%%%%%%%%%
%%%%%%%%%%%%%%%%%%%%%%%%%%%%%%%%%%%%%%%%%%%%%%%%%%%%%%%%%%%%%%%%%%%%%%

In this section we investigate the mollifier for functions in $L_{p,\phi}(\R^d)$.

%--------------------------------------------------------
Let $\eta$ be a function on $\R^d$ 
such that
\begin{equation}\label{eta}
 \supp\eta\subset\overline{B(0,1)},
 \quad
 0\le\eta\le2
 \quad\text{and}\quad
 \int_{B(0,1)}\eta(y)\,dy=|B(0,1)|,
\end{equation}
and let
$\eta_t(x)=|B(0,t)|^{-1}\eta(x/t)$.
Then, for $f\in L^1_{\loc}(\R^d)$,
\begin{equation}\label{mollifier}
 \eta_t*f(x)=\fint_{B(x,t)}\eta((x-y)/t)f(y)\,dy.
\end{equation}

%%%==============================================
\begin{thm}\label{thm:mollifier}
%%%==============================================
Let $\eta$ satisfy \eqref{eta}. 
Let $p\in[1,\infty)$ and let $\phi$ be almost decreasing. 
Then there exists a positive constant $C$ such that,
for all $f\in L_{p,\phi}(\R^d)$ and $t>0$,
\begin{equation*}
 \|\eta_t*f\|_{L_{p,\phi}}\le C\|f\|_{L_{p,\phi}},
\end{equation*}
where the constant $C$ is dependent only on $d$, $p$ and $\phi$.
\end{thm}

To prove the above theorem 
we state the basic inequality on the mean value and show two lemmas.

Let $p\in[1,\infty)$.
For two balls $B_1$ and $B_2$,
if $B_1\subset B_2$, then
\begin{equation}\label{B1B2}
 \M_p(f,B_1)
 \le
 \left(\frac{|B_2|}{|B_1|}\right)^{1/p}\M_p(f,B_2).
\end{equation}

%=================================================
\begin{lem}\label{lem:(ef)B r<t}
%=================================================
Let $\eta$ satisfy \eqref{eta}. 
Let $1\le p<\infty$, 
and let $\phi$ be almost decreasing. 
If $0<r\le t$, then, for any $x\in\R^d$, 
\begin{equation}
 \frac{\M_p(\eta_{t}*f,B(x,r))}{\phi(x,r)}
 \le
 C\frac{\M_p(f,B(x,2t))}{\phi(x,2t)},
\end{equation}
where the constant $C$ is dependent only on $d$, $p$ and $\phi$.
\end{lem}

\begin{proof}
If $y\in B(x,r)$ and $r\le t$, then $B(y,t)\subset B(x,2t)$ and
\begin{align*}
 |\eta_{t}*f(y)| 
 &=
 \bigg|\fint_{B(y,t)}\big(\eta((y-z)/t)\big)f(z)\,dz\bigg| \\
 &\le
 2^{d/p}\left(\fint_{B(x,2t)}\big|\big(\eta((y-z)/t)\big)f(z)\big|^p\,dz\right)^{1/p} \\
 &\le
 2\cdot2^{d/p}
  \left(\fint_{B(x,2t)}|f(z)|^p\,dz\right)^{1/p}.
\end{align*}
From this and the almost decreasingness of $\phi$ it follows that
\begin{equation*}
 \frac1{\phi(x,r)}\left(\fint_{B(x,r)}|\eta_{t}*f(y)|^p\,dy\right)^{1/p}
 \le
 C\frac{\M_p(f,B(x,2t))}{\phi(x,2t)}.
 \qedhere
\end{equation*}
\end{proof}

%=================================================
\begin{lem}\label{lem:(ef)B t<r}
%=================================================
Let $\eta$ satisfy \eqref{eta} and let $1\le p<\infty$.
If $0<t\le r$, then, for any $x\in\R^d$, 
\begin{equation}\label{eta*f}
 \bigg(\fint_{B(x,r)}\big|\eta_{t}*f(y)\big|^p\,dy\bigg)^{1/p} 
 \le
 2^{d/p}\M_p(f,B(x,2r)).
\end{equation}
Consequently,
\begin{equation}\label{f-eta*f}
 \bigg(\fint_{B(x,r)}\big|f(y)-\eta_{t}*f(y)\big|^p\,dy\bigg)^{1/p} 
 \le
 2\cdot2^{d/p}\M_p(f,B(x,2r)). 
\end{equation}
\end{lem}

\begin{proof}
Let $B=B(x,r)$. 
If $y\in B$ and $t\le r$, then $B(y,t)\subset 2B$ and
\begin{align*}
 \eta_{t}*f(y)
 &=
 \fint_{B(y,t)} \eta\left(\frac{y-z}{t}\right)f(z)\,dz \\
 &=
 \fint_{B(y,t)} \eta\left(\frac{y-z}{t}\right)(f\chi_{2B})(z)\,dz \\
 &=
 \eta_t*(f\chi_{2B})(y).
\end{align*}
Hence,
\begin{align*}
 \bigg(\fint_{B}\big|\eta_{t}*f(y)\big|^p\,dy\bigg)^{1/p} 
 &\le
 \frac1{|B|^{1/p}}\|\eta_t*(f\chi_{2B})\|_{L^p}
 \le
 \frac1{|B|^{1/p}}\|\eta_t\|_{L^1}\|f\chi_{2B}\|_{L^p} \\
 &=
 2^{d/p}\bigg(\fint_{2B}\big|f(y)\big|^p\,dy\bigg)^{1/p} 
 =
 2^{d/p}\M_p(f,2B).
\end{align*}
This shows \eqref{eta*f},
from which it follows that
\begin{align*}
 \bigg(\fint_{B}\big|f(y)-\eta_{t}*f(y)\big|^p\,dy\bigg)^{1/p} 
 &\le
 \bigg(\fint_{B}\big|f(y)\big|^p\,dy\bigg)^{1/p}
  +\bigg(\fint_{B}\big|\eta_{t}*f(y)\big|^p\,dy\bigg)^{1/p} \\
 &\le
 2\cdot 2^{d/p}\M_p(f,2B). 
\qedhere 
\end{align*}
\end{proof}
%%%---------------------------------------------------

%%%---------------------------------------------------
\begin{proof}[Proof of Theorem~\ref{thm:mollifier}]
For any ball $B(x,r)$,
if $r\le t$, then, using Lemma~\ref{lem:(ef)B r<t}, we have
\begin{equation*}
 \frac{\M_p(\eta_{t}*f,B(x,r))}{\phi(x,r)}
 \le
 C\|f\|_{L_{p,\phi}}. 
\end{equation*}
If $t\le r$, then, using \eqref{eta*f} in Lemma~\ref{lem:(ef)B t<r} and 
the inequality $\phi(x,2r)\ls \phi(x,r)$, we have
\begin{equation*}
 \frac{\M_p(\eta_{t}*f,B(x,r))}{\phi(x,r)}
 \le
 \frac{2\cdot2^{d/p}\M_p(f,B(x,2r))}{\phi(x,r)}
 \le
 C\|f\|_{L_{p,\phi}}. 
 \qedhere
\end{equation*}
\end{proof}

%%%%%%%%%%%%%%%%%%%%%%%%%%%%%%%%%%%%%%%%%%%%%%%%%%%%%%%%%%%%%%%%%%%%%%
%%%%%%%%%%%%%%%%%%%%%%%%%%%%%%%%%%%%%%%%%%%%%%%%%%%%%%%%%%%%%%%%%%%%%%
\section{Proof of Theorem~\ref{thm:distance}}\label{s:distance}
%%%%%%%%%%%%%%%%%%%%%%%%%%%%%%%%%%%%%%%%%%%%%%%%%%%%%%%%%%%%%%%%%%%%%%
%%%%%%%%%%%%%%%%%%%%%%%%%%%%%%%%%%%%%%%%%%%%%%%%%%%%%%%%%%%%%%%%%%%%%%

We divide the proof of Theorem~\ref{thm:distance}
into two parts, that is, 
$d(f,\Cic(\R^d))\le CA_{p,\phi}$ and $A_{p,\phi}\le d(f,\Cic(\R^d))$.
In this section we always assume that $p\in[1,\infty)$ and $\phi$ is in $\cGdec_p$ and satisfie \eqref{NC}.

%%%%%%%%%%%%%%%%%%%%%%%%%%%%%%%%%%%%%%%%%%%%%%%%%%%%%%%%%%%%%%%%%%%%%%
\subsection{Proof of $d(f,\Cic(\R^d))\le CA_{p,\phi}(f)$}
%%%%%%%%%%%%%%%%%%%%%%%%%%%%%%%%%%%%%%%%%%%%%%%%%%%%%%%%%%%%%%%%%%%%%%

In this subsection we prove this inequality without \eqref{lim1} and \eqref{lim2}.

Let $f\in L_{p,\phi}(\R^d)$. 
For any $\epsilon>0$,
from the definition of $A_{p,\phi}(f)$ there exist integers $i_{\epsilon}$ and $k_{\epsilon}$ ($i_{\epsilon}<k_{\epsilon}$) 
such that
\begin{equation*}
\sup\left\{\frac{\M_p(f,B(x,r))}{\phi(x,r)}:x\in\R^d,\ 0<r\le2^{i_{\epsilon}}\right\}<A_{p,\phi}(f)+\epsilon, 
\end{equation*}
\begin{equation*}
\sup\left\{\frac{\M_p(f,B(x,r))}{\phi(x,r)}:x\in\R^d,\ r\ge2^{k_{\epsilon}}\right\}<A_{p,\phi}(f)+\epsilon 
\end{equation*}
and 
\begin{equation*}
 \limsup_{|x|\to\infty}
 \max\left\{\frac{\M_p(f,B(x,2^{\ell}))}{\phi(x,2^{\ell})} :{\ell=i_{\epsilon}, i_{\epsilon}+1, \dots,\,k_{\epsilon}}\right\}
  \le A_{p,\phi}(f). %=0.
\end{equation*}
By \eqref{B1B2} and \eqref{DC} on $\phi$ we have
\begin{equation*}
 \sup_{2^{\ell-1}\le r\le2^{\ell}}\frac{\M_p(f,B(x,r))}{\phi(x,r)}
 \le C_1\frac{\M_p(f,B(x,2^{\ell}))}{\phi(x,2^{\ell})},
 \quad
 \ell=i_{\epsilon}, i_{\epsilon}+1,\dots, k_{\epsilon},
\end{equation*}
where the positive constant $C_1\ge1$ is dependent only on $d$, $p$ and $\phi$.
Consequently,
\begin{equation*}
 \limsup_{|x|\to\infty}\sup_{2^{i_{\epsilon}}\le r\le2^{k_{\epsilon}}}
  \frac{\M_p(f,B(x,r))}{\phi(x,r)} \le C_1A_{p,\phi}(f). %=0.
\end{equation*}
Then there exists an integer $j_{\epsilon}$ such that 
$j_{\epsilon}>k_{\epsilon}(>i_{\epsilon})$ 
and
\begin{equation*}
 \sup\left\{\frac{\M_p(f,B(x,r))}{\phi(x,r)}:B(x,r)\cap B(0,2^{j_{\epsilon}})=\emptyset\right\}< C_1A_{p,\phi}(f)+\epsilon.
\end{equation*}
Using $i_{\epsilon}$, $k_{\epsilon}$ and $j_{\epsilon}$,
we set
\begin{align*}
 \cB_1&=\left\{B(x,r):x\in\R^d,\ 0<r\le2^{i_{\epsilon}}\right\}, \\
 \cB_2&=\left\{B(x,r):x\in\R^d,\ r\ge2^{k_{\epsilon}}\right\}, \\
 \cB_3&=\left\{B(x,r):B(x,r)\cap B(0,2^{j_{\epsilon}})=\emptyset\right\}.
\end{align*}
Then %$\M_p(f,B)/\phi(B)<C_1A_{p,\phi}(f)+\epsilon$ if $B\in\cB_1\cup\cB_2\cup\cB_3$.
\begin{equation}\label{B123}
 \frac{\M_p(f,B)}{\phi(B)}<C_1A_{p,\phi}(f)+\epsilon,
 \quad\text{if}\quad 
 B\in\cB_1\cup\cB_2\cup\cB_3.
\end{equation}

Next, let $s=2^{j_\epsilon+2}$, and let $f_1=f\chi_{B(0,s)}$. 
If $r>2^{j_\epsilon}$ or $B(x,r)\cap B(0,2^{j_\epsilon})=\emptyset$, 
then $B(x,r)\in\cB_2\cup\cB_3$. 
By the fact that $|f-f_1|\le|f|$ and \eqref{B123} we have
\begin{equation*}
\frac{\M_p(f-f_1,B(x,r))}{\phi(x,r)}
\le
\frac{\M_p(f,B(x,r))}{\phi(x,r)} 
\le
C_1A_{p,\phi}(f)+\epsilon.
\end{equation*}
If $r\le2^{j_\epsilon}$ and $B(x,r)\cap B(0,2^{j_\epsilon})\neq\emptyset$, 
then $f-f_1=0$ on $B(x,r)\subset B(0,s)$. 
Therefore, 
\begin{equation}\label{f-f1}
 \|f-f_1\|_{L_{p,\phi}}
 \le
 C_1A_{p,\phi}(f)+\epsilon. 
\end{equation} 
By the fact that $|f_1|\le|f|$ and \eqref{B123} we also have 
\begin{equation}\label{M f1}
 \frac{\M_p(f_1,B)}{\phi(B)}
 \le
 \frac{\M_p(f,B)}{\phi(B)} 
 \le
 C_1A_{p,\phi}(f)+\epsilon
 \quad\text{for}\quad
 B\in\cB_1\cup\cB_2\cup\cB_3.
\end{equation}

Next, let 
\begin{equation*}
C_\epsilon
=
\inf\left\{ |B(x,r)|^{1/p}\phi(x,r):2^{i_\epsilon-1}\le r\le2^{k_\epsilon-1}, B(x,r)\subset B(0,s)\right\}.
\end{equation*}
Recall that $s=2^{j_{\epsilon}+2}$.
Then $C_\epsilon>0$,
since 
\begin{equation*}
|B(x,r)|^{1/p}\phi(x,r)
\sim 
r^{d/p}\phi(x,s)
\sim 
r^{d/p}\phi(0,s)
\ge
2^{(i_\epsilon-1)d/p}\phi(0,s),
\end{equation*} 
if $2^{i_\epsilon-1}\le r\le2^{k_\epsilon-1}$ and $B(x,r)\subset B(0,s)$.
In the above we used \eqref{DC} and \eqref{NC} on $\phi$.
Let $\eta$ be in $\Cic(\R^d)$ and satisfy \eqref{eta}.
Since $\eta_{r}*f_1\to f_1$ in $L^p(B(0,s))$ as $r\to+0$,
we can choose $t>0$ such that $t<2^{i_\epsilon-1}$ and that
\begin{equation}\label{C_epsilon}
\left(\int_{B(0,s)} |f_1(y)-\eta_{t}*f_1(y)|^p\,dy\right)^{1/p}<C_\epsilon\epsilon.
\end{equation}
Let $f_2=\eta_{t}*f_1$.
Then $f_2\in\Cic(\R^d)$, that is,
\begin{equation}\label{f2-g}
 \min_{g\in\Cic(\R^d)}\|f_2-g\|_{L_{p,\phi}}=0.
\end{equation}
In the following, using \eqref{M f1}, we will show that
\begin{equation}\label{f1-f2}
 \|f_1-f_2\|_{L_{p,\phi}}
 \le
 C_2(C_1A_{p,\phi}(f)+\epsilon).
\end{equation}
Once we show \eqref{f1-f2}, 
combining this with \eqref{f-f1} and \eqref{f2-g},
we obtain that 
\begin{equation*}
 d(f,\Cic(\R^d))=\inf_{g\in\Cic(\R^d)}\|f-g\|_{L_{p,\phi}}\le (1+C_2)C_1A_{p,\phi}(f). 
\end{equation*}

Now, take a ball $B=B(x,r)$ arbitrarily. 

\noindent
Case 1: 
If $2B\in\cB_1\cup\cB_2\cup\cB_3$, then $B\in\cB_1\cup\cB_3$ or $r>t$. 
From Lemmas~\ref{lem:(ef)B r<t} and \ref{lem:(ef)B t<r} and \eqref{DC} on $\phi$ it follows that
\begin{align*}
 \frac{\M_p(f_1-f_2,B)}{\phi(B)}
 &\le
 \begin{cases}
  \dfrac{\M_p(f_1,B)}{\phi(B)}
  +
  C_3\dfrac{\M_p(f_1,B(x,2t))}{\phi(x,2t)}& (r\le  t) \medskip
 \\
  2\cdot 2^{d/p}\dfrac{\M_p(f_1,2B)}{\phi(B)} & (r>t)
\end{cases}
\\
 &\le
 C'_2(C_1A_{p,\phi}(f)+\epsilon). 
\end{align*}
In the above, if $r\le t$, then $B$ and $B(x,2t)$ are in $\cB_1$, since $t<2^{i_{\epsilon}-1}$. 
Otherwise, $2B\in\cB_1\cup\cB_2\cup\cB_3$ by the assumption.
Then we have the last inequality by \eqref{M f1}.

\noindent
Case 2: 
If $2B\notin\cB_1\cup\cB_2\cup\cB_3$, 
then $2^{i_\epsilon-1}\le r\le2^{k_\epsilon-1}$ and $B\subset B(0,s)$.
Hence, by \eqref{C_epsilon} we have 
\begin{align*}
\frac{\M_p(f_1-f_2,B)}{\phi(B)}
&=
\frac{1}{|B|^{1/p}\phi(B)}\left(\int_B|f_1(y)-f_2(y)|^p\,dy\right)^{1/p} \\
&\le
\frac{1}{C_\epsilon}\left(\int_{B(0,s)}|f_1(y)-f_2(y)|^p\,dy\right)^{1/p}
<\epsilon.
\end{align*} 
%We let $C_2=\max(C'_2,1/|B(0,1)|^{1/p})$.
Then we have \eqref{f1-f2} and the proof is complete.

\subsection{Proof of $A_{p,\phi} \le d(f,\Cic(\R^d))$}
%%%%%%%%%%%%%%%%%%%%%%%%%%%%%%%%%%%%%%%%%%%%%%%%%%%%%%%%%%%%%%%%%%%%%%
Assume \eqref{lim1} and \eqref{lim2}.
Firstly, let $f\in\Cic(\R^d)$ and $\supp f \subset B(0,M)$. 
For a ball $B(x,r)$, if $r<M$ and $B(x,r)\cap B(0,M)\ne\emptyset$,
then $x \in B(0,2M)$. 
Hence,
\begin{equation*}
 \lim_{r\to+0}\sup_{x\in\R^d}\frac{\M_p(f,B(x,r))}{\phi(x,r)}
 =
 \lim_{r\to+0}\sup_{x\in B(0,2M)}\frac{\M_p(f,B(x,r))}{\phi(x,r)}.
\end{equation*}
From \eqref{lim1}
it follows that
\begin{equation*}
 \dlim_{r\to+0}\sup_{x\in\R^d}\frac{\M_p(f,B(x,r))}{\phi(x,r)}
 \le
 \dlim_{r\to+0}\sup_{x\in B(0,2M)}\frac{\|f\|_{L^\infty}}{\phi(x,r)}
 =0.
\end{equation*}
On the other hand, 
if $r\ge M$ and $B(x,r)\cap B(0,M)\ne\emptyset$, 
then $x\in B(0,2r)$.
Hence,
\begin{equation*}
 \lim_{r\to\infty}\sup_{x\in\R^d}\frac{\M_p(f,B(x,r))}{\phi(x,r)}
 =
 \lim_{r\to\infty}\sup_{x\in B(0,2r)}\frac{\M_p(f,B(x,r))}{\phi(x,r)}.
\end{equation*}
If $x\in B(0,2r)$, then $\phi(x,r)\sim\phi(x,2r)\sim\phi(0,2r)\sim\phi(0,r)$ by \eqref{DC} and \eqref{NC} on $\phi$,
which implies
\begin{equation*}
 \frac{\M_p(f,B(x,r))}{\phi(x,r)}
 \le
 \frac{\|f\|_{L^p}}{\phi(x,r)|B(x,r)|^{1/p}}
 \sim
 \frac{\|f\|_{L^p}}{r^{d/p}\phi(0,r)}.
\end{equation*}
Hence, from \eqref{lim2} it follows that
\begin{equation*}
 \dlim_{r\to\infty}\sup_{x\in\R^d}\frac{\M_p(f,B(x,r))}{\phi(x,r)}
 \ls
 \dlim_{r\to\infty}\frac{\|f\|_{L^p}}{r^{d/p}\phi(0,r)}
 =0.
\end{equation*}
For each $r>0$, take $x\in\R^d$ such that $\supp f\cap B(x,r)=\emptyset$.
Then
\begin{equation*}
 \frac{\M_p(f,B(x,r))}{\phi(x,r)}=0.
\end{equation*}
That is, %$f$ satisfies \ref{rto0}--\ref{xtoinfty} in Corollary~\ref{cor:closure}. 
$A_{p,\phi}(f)=0$.

Let $f\in L_{p,\phi}(\R^d)$.
Then, for any $\epsilon>0$, there exists $g\in\Cic(\R^d)$ such that 
$\|f-g\|_{L_{p,\phi}}<d(f,\Cic(\R^d))+\epsilon$.
That is, 
\begin{align*}
\frac{\M_p(f,B(x,r))}{\phi(x,r)}
&\le
\frac{\M_p(g,B(x,r))}{\phi(x,r)}
+
\|f-g\|_{L_{p,\phi}}\\
&<
\frac{\M_p(g,B(x,r))}{\phi(x,r)}
+
d(f,\Cic(\R^d))+\epsilon,
\end{align*}
for all balls $B(x,r)$.
Therefore, we have %the first inequality in Theorem~\ref{thm:distance}. 
\begin{equation*}
 A_{p,\phi}(f)
 \le A_{p,\phi}(g)+d(f,\Cic(\R^d))+\epsilon
 =d(f,\Cic(\R^d))+\epsilon.
\end{equation*}
This shows the conclusion.

%%%%%%%%%%%%%%%%%%%%%%%%%%%%%%%%%%%%%%%%%%%%%%%%%%%%%%%%%%%%%%%%%%%%%%
%%%%%%%%%%%%%%%%%%%%%%%%%%%%%%%%%%%%%%%%%%%%%%%%%%%%%%%%%%%%%%%%%%%%%%
\section{Proof of Theorem~\ref{thm:bidual}}\label{s:duality}
%%%%%%%%%%%%%%%%%%%%%%%%%%%%%%%%%%%%%%%%%%%%%%%%%%%%%%%%%%%%%%%%%%%%%%
%%%%%%%%%%%%%%%%%%%%%%%%%%%%%%%%%%%%%%%%%%%%%%%%%%%%%%%%%%%%%%%%%%%%%%

%In this section we prove Theorem~\ref{thm:bidual}. 
First we state the definition of the block space $B^{[\phi,q]}(\R^d)$.

%%%=======================
\begin{defn}[{$[\phi,q]$-block}]  \label{defn:block}
%%%=======================
Let $\phi:\R^d\times(0,\infty)\to(0,\infty)$
and $1<q\le\infty$.
A function $b$ on $\R^d$ is called a $[\phi,q]$-block 
if there exists a ball $B$ 
such that
\begin{enumerate}
\item \label{supp}
 $\supp b \subset {B}$, 
\item \label{size}
 $\displaystyle
 \|b\|_{L^q} \le \frac{1}{|B|^{1/q'}\phi(B)}$,
\end{enumerate}
where 
$\|b\|_{L^q}$ is the $L^q$ norm of $b$ 
and $1/q+1/q'=1$. 
\end{defn} 
We call $B$ the corresponding ball to $b$. 
We also denote by $B[\phi,q]$ the set of all $[\phi,q]$-blocks. 

%%%----------------------------------------------------------
If $b$ is a $[\phi,q]$-block and a ball $B$ satisfies \ref{supp} and \ref{size},
then 
\begin{align*}
   \left|\int_{\R^d} b(x)g(x)\,dx\right|
 & \le
   \|b\|_{L^q} \left(\int_B |g(x)|^{q'}\,dx\right)^{1/q'} \\
 & \le
   \frac{1}{\phi(B)}
     \left(\frac{1}{|B|}\int_B |g(x)|^{q'}\,dx\right)^{1/q'} \\
 & \le
   \|g\|_{L_{q',\phi}}. 
\end{align*}
That is, 
the mapping $g\mapsto\int bg$ 
is a bounded linear functional on $L_{q',\phi}(\R^d)$
with norm not exceeding $1$.

%%%----------------------------------------------------------

%%%=======================
\begin{defn}[$B^{[\phi,q]}(\R^d)$]     \label{defn:B^Phi}     
%%%=======================
Let $\phi:\R^d\times(0,\infty)\to(0,\infty)$, 
$1<q\le\infty$ and $1/q+1/q'=1$.
Assume that $L_{q',\phi}(\R^d) \not=\{0\}$.
We define the space 
$B^{[\phi,q]}(\R^d)\subset(L_{q',\phi}(\R^d))^*$ 
as follows:
\begin{quote}
$f\in B^{[\phi,q]}(\R^d)$
if and only if 
there exist sequences $\{b_j\}\subset B[\phi,q]$
and positive numbers $\{\lambda_j\}$
such that
\begin{equation}          \label{expression0}
     f=\sum_j \lambda_j b_j \;\text{in}\;(L_{q',\phi}(\R^d))^*
     \quad\text{and}\quad
     \sum_j\lambda_j<\infty.
\end{equation}
\end{quote}
\end{defn}

In general, the expression \eqref{expression0} is not unique.
We define
\begin{equation*}
     \|f\|_{B^{[\phi,q]}}=
     \inf\left\{\sum_j\lambda_j\right\},
\end{equation*}
where the infimum is taken over all expressions as in \eqref{expression0}.
Then $\|f\|_{B^{[\phi,q]}}$ is a norm
and 
$B^{[\phi,q]}(\R^d)$ is a Banach space.

%%%----------------------------------------------------------
Next we state the known results.
The following two propositions give sufficient conditions of $L_{q',\phi}(\R^d)\ne\{0\}$
and $B^{[\phi,q]}(\R^d)\subset L^1_{\loc}(\R^d)$.
We give their proofs for convenience.

%%%======================
\begin{prop}[\cite{Tang-Nakai-Yang-Zhu-preparation}]\label{prop:Cic-gM}
%%%======================
Let $p\in[1,\infty)$, and let $\phi$ be in $\cGdec_p$ and satisfy \eqref{NC}. 
Then, for any ball $B$, 
$\chi_B\in L_{p,\phi}(\R^d)$ and $\|\chi_B\|_{L_{p,\phi}}
\le C/\phi(B)$, 
where the positive constant $C$ is independent of $B$. 
Consequently, $\Cic(\R^d)\subset L_{p,\phi}(\R^d)$. 
\end{prop}

\begin{proof}
Let $B=B(x,r)$. 
We show that, for all balls $B(y,s)$, 
\begin{equation*}
 \frac{\M_p(\chi_{B},B(y,s))}{\phi(y,s)} \le C_\phi\,\phi(B). 
\end{equation*}
We may assume that $B(x,r) \cap B(y,s) \ne \emptyset$.
If $0<r<s$, then $|x-y|\le2s$. 
From the almost increasingness of $s\mapsto r^{d/p}\phi(y,s)$, 
the conditions \eqref{DC} and \eqref{NC} on $\phi$ 
it follows that 
\begin{equation}\label{rxr sys}
r^{d/p}\phi(B)
=
r^{d/p}\phi(x,r)
\ls
s^{d/p}\phi(x,s)
\sim
s^{d/p}\phi(y,s). 
\end{equation}
Thus
\begin{equation*}
\frac{\M_p(\chi_{B},B(y,s))}{\phi(y,s)}
\le
\frac{|B|^{1/p}}{|B(y,s)|^{1/p}\phi(y,s)}
=
\frac{r^{n/p}}{s^{d/p}\phi(y,s)}
\le
C_\phi\frac{1}{\phi(B)}. 
\end{equation*} 
If $r\ge s$, then $|x-y|\le2r$. 
From the conditions \eqref{DC}, \eqref{NC} 
and the almost decreasingness of $\phi$ 
it follows that 
\begin{equation}\label{xr ys}
\phi(B)=\phi(x,r)
\sim
\phi(x,2r)
\sim
\phi(y,2r)
\ls
\phi(y,s). 
\end{equation}
Thus
\begin{equation*}
\frac{\M_p(\chi_{B},B(y,s))}{\phi(y,s)}
\le
\frac{1}{\phi(y,s)}
\le
C_\phi\frac{1}{\phi(B)}. 
\end{equation*} 
The proof is complete. 
\end{proof}

%Note that, if $\displaystyle\inf_j\phi(B_j)>0$, then the expression in \eqref{expression0} converges in $L^1(\R^d)$, 
%where $B_j$ is the corresponding ball to $b_j$ for each $j$. 
%For sufficient conditions of $\Cic(\R^d)\subset L_{p,\phi}(\R^d)$, 
%see Proposition~\ref{prop:Cic-gM}. 

%%%======================
\begin{prop}[\cite{Tang-Nakai-Yang-Zhu-preparation}]\label{prop:L 1 loc}
%%%======================
Let $p\in[1,\infty)$, $q\in(1,\infty]$ and $1/p+1/q=1$.
Let $\phi$ be in $\cGdec_p$ and satisfy \eqref{NC}. 
Then $B^{[\phi,q]}(\R^d)\subset L^1_{\loc}(\R^d)$
and the expression in \eqref{expression0} converges a.e.
\end{prop}

\begin{proof}
Let $\{b_j\}\subset B[\phi,q]$, $\lambda_j\ge0$ and $\sum_{j}\lambda_j<\infty$.
For each $j$, let $B_j$ be the supporting ball of $b_j$.
Then,
for any ball $B$,
\begin{align*}
 \int_B \sum_j \lambda_j |b_j|
 &\le
 \sum_j \lambda_j \int_{B\cap B_j} |b_j|
 \le
 \sum_j \lambda_j \|\chi_{B\cap B_j}\|_{L^p(\R^d)} \|b_j\|_{L^q(\R^d)} \\
 &\le
 \sum_j \lambda_j \frac{|B\cap B_j|^{1/p}}{|B_j|^{1/p}\phi(B_j)}.
% \lesssim
% \sum_j \lambda_j \frac{1}{\phi(B)}<\infty.
\end{align*}
For any ball $B$,
if $B\cap B_j\ne\emptyset$, then 
$|B|^{1/p}\phi(B)\ls|B_j|^{1/p}\phi(B_j)$ or 
$\phi(B)\ls\phi(B_j)$ by \eqref{rxr sys} and \eqref{xr ys}. 
If $|B|^{1/p}\phi(B)\ls|B_j|^{1/p}\phi(B_j)$, then
\begin{equation*}
 \frac{|B\cap B_j|^{1/p}}{|B_j|^{1/p}\phi(B_j)}
 \le
 \frac{|B|^{1/p}}{|B_j|^{1/p}\phi(B_j)}
 \ls
 \frac{1}{\phi(B)}.
\end{equation*}
If $\phi(B)\ls\phi(B_j)$, then
\begin{equation*}
 \frac{|B\cap B_j|^{1/p}}{|B_j|^{1/p}\phi(B_j)}
 \le
 \frac{1}{\phi(B_j)}
 \ls
 \frac1{\phi(B)}.
\end{equation*}
Then $\sum_j \lambda_j b_j$ converges a.e. to 
some function in $L^1_{\mathrm{loc}}(\R^d)$.
\end{proof}

%%%----------------------------------------------------------
Next we state the known duality.

%%%-----------------------------------------------------------
%%%=======================
\begin{thm}[\cite{Nakai2008AMS}]\label{thm:duality}
%%%=======================
Let $p,q\in(1,\infty)$ and $1/p+1/q=1$.
Assume that $\phi,q$ satisfy the conditions in Definition~\ref{defn:B^Phi}.
Then
\begin{equation*}
     \left(B^{[\phi,q]}(\R^d)\right)^*=L_{p,\phi}(\R^d).
\end{equation*}
More precisely, if $g\in L_{p,\phi}(\R^d)$,
then the mapping $\ell_g : f \mapsto \int_{\R^d} fg$, 
for $f\in L^q_\comp(\R^d)$,  
can be extended to a bounded linear functional on $B^{[\phi,q]}(\R^d)$.
Conversely, 
if $\ell$ is a bounded linear functional on $B^{[\phi,q]}(\R^d)$, 
then there exists $g \in L_{p,\phi}(\R^d)$
such that $\ell(f)=\ell_g(f)$ for $f\in L^q_\comp(\R^d)$. 
The norm $\|\ell\|_{(B^{[\phi,q]})^*}$ is equivalent to $\|g\|_{L_{p,\phi}}$.
\end{thm}
%%%-----------------------------------------------------------

The above theorem is valid for spaces of homogeneous type. 
For Morrey space with non doubling measure, 
see Sawano and Tanaka~\cite{Sawano-Tanaka2009}. 
There are related results in 
\cite{Gogatishvili-Mustafayev2013, 
Maeda-Mizuta-Ohno-Shimomura2019, 
Mizuta2016, 
Ono2023}. 

To prove Theorem~\ref{thm:bidual} 
it is enough to prove the following theorem. 

%=================================================
\begin{thm}\label{thm:dual}
%=================================================
Let $p,q\in(1,\infty)$ and $1/p+1/q=1$, 
and let $\phi$ be in $\cGdec_p$ 
and satisfy \eqref{lim1}, \eqref{lim2} and \eqref{NC}. 
Then
\begin{equation}\label{CgM-B}
 \Big(\overline{\Cic(\R^d)}^{L_{p,\phi}(\R^d)}\Big)^*=B^{[\phi,q]}(\R^d).
\end{equation}
More precisely, for $f\in B^{[\phi,q]}(\R^d)$,
the linear functional 
\begin{equation}\label{dual}
 \langle f,v \rangle = \int_{\R^d}f(x)v(x)\,dx,  \quad v\in\Cic(\R^d)
\end{equation}
can be extended on $\overline{\Cic(\R^d)}^{L_{p,\phi}(\R^d)}$.
Conversely, 
each bounded linear functional on $\overline{\Cic(\R^d)}^{L_{p,\phi}(\R^d)}$
has the form \eqref{dual} 
for some $f\in B^{[\phi,q]}(\R^d)$.
The linear functional norm 
is equivalent to $\|f\|_{B^{[\phi,q]}}$.
\end{thm}

By the condition \eqref{NC} on $\phi$, the condition \eqref{lim2} is equivalent to
\begin{equation*}
 \lim_{r\to\infty}r^{d/p}\phi(x,r)=\infty 
 \quad 
 \text{for each $x\in\R^d$}.
\end{equation*}

%To prove the above theorem 
%we use the method by Coifman and Weiss \cite{Coifman-Weiss1977}.
%Our proof is similar to \cite{Yamaguchi-Nakai-Shimomura2023AMS}.
%However, some parts are different from \cite{Yamaguchi-Nakai-Shimomura2023AMS}
%since it dealt with the function spaces modulo constants.

First we recall two  known theorems in functional analysis.

%%%==============================================
\begin{thm}[{cf. \cite[Exercise~41 (p.~439)]{Dunford-Schwartz1958}}]
\label{thm:total dense}
%%%==============================================
Let $X$ be a locally convex linear topological space, 
and let $X^*$ be its dual space.
If $T$ is a linear subspace of $X^*$,
then $T$ is dense in the weak* topology 
if and only if 
$T$ is a total set of functionals on $X$.
\end{thm}

%%%==============================================
\begin{thm}[{cf. \cite[Theorem~3.17]{Rudin1991}}]
\label{thm:seq Banach-Alaoglu}
%%%==============================================
Let $X$ be a separable normed vector space. 
If $\{\Lambda_n\}$ is a bounded sequence in $X^*$,
then there exists a subsequence $\{\Lambda_{n_i}\}$
such that it converges to an element in $X^*$ 
in the weak* topology.
\end{thm}

Next we show two propositions and three lemmas. 
While the proofs of the following two propositions 
are similar to Propositions~5.5 and 5.6 
in \cite{Yamaguchi-Nakai-Shimomura2023AMS}, 
we give their proofs to clarify the differences.

%%%==============================================
\begin{prop}\label{prop:separable}
%%%==============================================
Let $p\in[1,\infty)$.
If $\phi$ is in $\cGdec_p$ and satisfies \eqref{NC}, 
then $\overline{\Cic(\R^d)}^{L_{p,\phi}(\R^d)}$
is a separable space.
\end{prop}

\begin{proof}
For $M>1$, let $C_M=\max(C_M^{1},C_M^{2})$ and
\begin{align}
 C_M^{1}
 &=\sup\left\{\frac{1}{\phi(x,r)}:x\in B(0,3M),\,0<r<1\right\},
                                                                                \label{0<r<1}\\
 C_M^{2}
 &=\sup\left\{\frac{(2M)^{d/p}}{r^{d/p}\phi(x,r)}:B(0,2M)\cap B(x,r)\ne\emptyset,\,r\ge1 \right\}.
                                                                                \label{r ge 1}
\end{align}
Then $C_M<\infty$, since $\phi$ is in $\cGdec_p$ and satisfies \eqref{NC}.

Now, let $h\in\Cic(\R^d)$ satisfy $\chi_{B(0,1)}\le h\le\chi_{B(0,2)}$
and let $h_R(x)=h(x/R)$, $R\in\N$.
We define 
\begin{align*}
  \cP_{\Q}
&=\bigcup_{R=1}^{\infty} \cP_{\Q,R}, \\
  \cP_{\Q,R}
&=\{p\,h_R : \text{$p$ is a trigonometric polynomial with rational coefficients}\}.
\end{align*}
It is enough to prove that, 
for all $f\in\Cic(\R^d)$ and $\epsilon>0$, 
there exists $g\in\cP_{\Q}$ 
such that $\|f-g\|_{L_{p,\phi}}<\epsilon$.
Let $\supp f \subset B(0,M)$. 
Then we can choose $g\in\cP_{\Q}$ 
such that $\supp g \subset B(0,2M)$ 
and that $\|f-g\|_{L^\infty}<\epsilon/C_M$. 
We show that, for any ball $B=B(x,r)$, 
\begin{equation*}
 \frac{\M_p(f-g,B)}{\phi(B)}<\epsilon.
\end{equation*}
If $0<r<1$, 
we may assume that $x\in B(0,3M)$. 
Otherwise, $\supp(f-g)\cap B(x,r)=\emptyset$. 
Thus
\begin{equation*}
 \frac {\M_p(f-g,B)}{\phi(B)} 
 \le \frac {\|f-g\|_{L^\infty}}{\phi(x,r)} 
 \le \epsilon.
\end{equation*}
If $r\ge1$,  
we may assume that $B(0,2M)\cap B(x,r)\ne\emptyset$. 
Thus
\begin{equation*}
 \frac {\M_p(f-g,B)}{\phi(B)} 
 \le
 \frac {\|f-g\|_{L^\infty}|B(0,2M)|^{1/p}}{\phi(B)|B|^{1/p}} 
 =
 \frac {\|f-g\|_{L^\infty}(2M)^{d/p}}{r^{d/p}\phi(x,r)} 
 \le \epsilon. 
\end{equation*}
Therefore, we get the conclusion. 
\end{proof}

%%%==============================================
\begin{prop}\label{prop:sup in CgC}
%%%==============================================
Let $p,q\in(1,\infty)$ and $1/p+1/q=1$, 
and let $\phi$ be in $\cGdec_p$ and satisfy \eqref{NC}.
Then, for every $f\in B^{[\phi,q]}(\R^d)$, 
\begin{equation}\label{sup in CgC}
 \|f\|_{B^{[\phi,q]}}
 \sim
 \sup\big\{\big|\ell_g(f)\big|:g\in\overline{\Cic(\R^d)}^{L_{p,\phi}(\R^d)}, 
 \ \|g\|_{L_{p,\phi}}\le1 \big\}, 
\end{equation}
where $\ell_g \in (B^{[\phi,q]}(\R^d))^*$ 
is the defined by Theorem~\ref{thm:duality}.
\end{prop}

\begin{proof} 
We may assume that $f\in L^q_\comp(\R^d)$, 
since $L^q_\comp(\R^d)$ is dense in $B^{[\phi,q]}(\R^d)$.
Then Theorem~\ref{thm:duality} shows that
\begin{align*}
 \|f\|_{B^{[\phi,q]}} 
 &=
 \sup\big\{\big|\ell(f)\big|:\ell\in(B^{[\phi,q]}(\R^d))^*, \|\ell\|_{(B^{[\phi,q]}(\R^d))^*}\le1 \big\} \\
 &\sim
 N(f)\equiv
 \sup\big\{\big|\ell_g(f)\big|:g\in L_{p,\phi}(\R^d), \|g\|_{L_{p,\phi}}\le1 \big\} \\
 &\ge
 \overline{N}(f)\equiv
 \sup\big\{\big|\ell_g(f)\big|:g\in\overline{\Cic(\R^d)}^{L_{p,\phi}(\R^d)}, 
           \ \|g\|_{L_{p,\phi}}\le1 \big\}.
\end{align*}
Next we show that $N(f)\ls\overline{N}(f)$.
For any $\epsilon>0$ there exists $g\in L_{p,\phi}(\R^d)$ 
such that $\|g\|_{L_{p,\phi}}\le1$ and $N(f)<|\ell_g(f)|+\epsilon$. 
We can choose a ball $B$ such that $\supp f\subset B$.
Then $\|g\chi_B\|_{L_{p,\phi}}\le 1$.
Let $\eta_t$ be as in Theorem~\ref{thm:mollifier}.
Then $(g\chi_B)*\eta_t\in\Cic(\R^d)$ 
and $\|(g\chi_B)*\eta_t\|_{L_{p,\phi}}\le C$ 
for some $C>0$ 
which is determined by only the constant
in Theorem~\ref{thm:mollifier}.
Moreover,
\begin{align*}
 |\ell_g(f)-\ell_{(g\chi_B)*\eta_t}(f)|
 &=
 |\ell_{g\chi_B}(f)-\ell_{(g\chi_B)*\eta_t}(f)| \\
 &=
 \Big|\int_{\R^d}(g\chi_B-(g\chi_B)*\eta_t)f \Big|\\
 &\le
 \|g\chi_B-(g\chi_B)*\eta_t\|_{L^p}\|f\|_{L^q}\to0
 \quad \text{as} \quad t\to+0.
\end{align*}
Thus, we can take $t>0$ such that $|\ell_g(f)-\ell_{(g\chi_B)*\eta_t}(f)|<\epsilon$.
Then
\begin{align*}
 |\ell_g(f)|
 &\le
 |\ell_{(g\chi_B)*\eta_t}(f)| + \epsilon \\
 &\le
 \sup\big\{\big|\ell_g(f)\big|:g\in\overline{\Cic(\R^d)}^{L_{p,\phi}(\R^d)}, 
 \ \|g\|_{L_{p,\phi}}\le C \big\}
 + \epsilon \\
 &=
 C\overline{N}(f)+\epsilon.
\end{align*}
Therefore,
\begin{equation*}
 N(f)\le C\overline{N}(f) + 2\epsilon. 
\end{equation*}
Since $\epsilon>0$ was arbitrary, the conclusion follows. 
\end{proof}

%-----------------------------------------------------------
The following lemma can be proved by the diagonalization argument. 
%-----------------------------------------------------------

%%%==============================================
\begin{lem}[{\cite[Lemma~4.3]{Coifman-Weiss1977}}]\label{lem:series}
%%%==============================================
Suppose $\lambda_j^k>0$, $j,k=1,2,\dots$, 
satisfies $\sum_{j=1}^{\infty}\lambda_j^k\le1$ for each $k=1,2,\dots$,
then there exists an increasing sequence of natural numbers 
$k_1<k_2<\dots<k_n<\cdots$ 
such that $\lim_{n\to\infty}\lambda_j^{k_n}=\lambda_j$ for each $j$
and that $\sum_{j=1}^{\infty}\lambda_j\le1$.
\end{lem}

%------------------------------------------------------------
For $n\in\Z$, let $\cB^n=\{B(3^nm,3^n\sqrt{d}):m\in\Z^d\}$ 
and enumerate the elements in $\cB^n$ by $\{B_i^n\}_{i=1}^\infty$. 
Then $\R^d$ can be covered with balls 
$B_1^n, B_2^n,\dots, B_i^n,\dots$ 
of radius $3^n\sqrt{d}$. 
Let $S_i^n=3B_i^{n-1}$.
Then, for each $n$,
$S_i^n \, (i=1,2,\dots)$ are balls of radius $3^n\sqrt{d}$
such that $\R^d=\bigcup_{i=1}^{\infty}S_i^n$ 
and that no point in $\R^d$ belongs to more than $N_d$ of these balls
\rm{(}$N_d$ depends only on $d$\rm{)}. 

In the above setting, we have the following lemma
in the same way as in \cite{Yamaguchi-Nakai-Shimomura2023AMS}. 

%%%==============================================
\begin{lem}[{cf. \cite[Lemma~5.8]{Yamaguchi-Nakai-Shimomura2023AMS}}]\label{lem:balls}
%%%==============================================
Let $q\in(1,\infty]$.
If $\phi$ satisfies the conditions \eqref{DC} and \eqref{NC},
then each $f\in B^{[\phi,q]}(\R^d)$ has the following representation:
\begin{equation*}
 f=\sum_{i=1}^{\infty}\sum_{n=-\infty}^{\infty}\lambda_i^nb_i^n,
\end{equation*}
where $b_i^n$ is a $[\phi,q]$-block supported in $S_i^n$,
$\lambda_i^n\ge0$ 
and 
$\sum_{i=1}^{\infty}\sum_{n=-\infty}^{\infty}\lambda_i^n
\le c\|f\|_{B^{[\phi,q]}}$.
\end{lem}

The following lemma can be proved by the same way as in \cite{Yamaguchi-Nakai-Shimomura2023AMS}.
However, we give its proof here since it is a key lemma to prove  Theorem~\ref{thm:dual}.
On the completion of the proof of Theorem~\ref{thm:dual},
from the following lemma 
it turns out that the unit ball of $B^{[\phi,q]}(\R^d)$ is weak* compact.

%%%==============================================
\begin{lem}\label{lem:subseq}
%%%==============================================
Let $p\in[1,\infty)$, $q\in(1,\infty]$ and $1/p+1/q=1$, 
and let $\phi$ be in $\cGdec_p$ and satisfy \eqref{NC}, \eqref{lim1} and \eqref{lim2}.
If $\|f_k\|_{B^{[\phi,q]}}\le1$, $k=1,2,\dots$, 
then there exist $f\in B^{[\phi,q]}(\R^d)$
and a subsequence $\{f_{k_j}\}$ such that
\begin{equation}\label{fv}
 \lim_{j\to\infty}\int_{\R^d}f_{k_j}(x)v(x)\,dx
 =\int_{\R^d}f(x)v(x)\,dx
\end{equation}
for all $v\in \Cic(\R^d)$.
\end{lem}

\begin{proof}
We write 
$f_k=\sum_{i,n}\lambda_i^{n,k}b_i^{n,k}$
as in Lemma~\ref{lem:balls}
with $\sum_{i,n}\lambda_i^{n,k}\le c$. 
We can suppose that the coefficients $\lambda_j^{n,k}$ are nonnegative
and, applying Lemma~\ref{lem:series}, 
we can assume $\lim_{k\to\infty}\lambda_i^{n,k}=\lambda_i^n$ 
exists for each $(i,n)$ and $\sum_{i,n}\lambda_i^n\le c$.
For each fixed $(i,n)$ the blocks $b_i^{n,k}$ are supported by
the $S_i^n$ for all $k$ and are uniformly bounded. 
Thus, $\{b_i^{n,k}\}_k$ is a bounded sequence in $L^q(S_i^n)$. 
By Proposition~\ref{prop:separable} and Theorem~\ref{thm:seq Banach-Alaoglu},
there exists a subsequence of $b_i^{n,1},b_i^{n,2},\dots,b_i^{n,k},\dots$
converging to a function $b_i^n$ in the weak* topology, 
that is
\begin{equation*}
 \int_{S_i^n} b_i^{n,k_j}(x)\vp(x)\,dx
 \to
 \int_{S_i^n} b_i^{n}(x)\vp(x)\,dx
\end{equation*}
for all $\vp\in L^p(S_i^n)$.
Applying, again, a diagonalization argument 
we can assume
$b_i^{n,k}\to b_i^n$ as $k\to\infty$ for all $(i,n)$.
Hence, we obtain $\supp b_i^n \subset S_i^n$ 
and $\|b_i^n\|_{L^q} \le 1/\{\phi(S_i^n)|S_i^n|^{1/p}\}$. 
Therefore, $b_i^n$ is a $[\phi,q]$-block.
We let $f=\sum_{i,n}\lambda_i^nb_i^n \in B^{[\phi,q]}(\R^d)$. 
It remains to be shown \eqref{fv} for all $v\in \Cic(\R^d)$.
We write
\begin{equation*}
 \int_{\R^d} f_k(x)v(x)\,dx
 =\Big(\sum_{-N\le n\le N}+\sum_{n<-N}+\sum_{n>N}\Big)
 \sum_i\lambda_i^{n,k}\int_{\R^d} b_i^{n,k}(x)v(x)\,dx.
\end{equation*}
Let $S_i^n=B(x_i^n,r_i^n)$ and consider the last sum. 
Without loss of generality we can assume $x_1^n=0$ for $n\in\Z$.
Thus $\supp v \subset S_1^N$ for large $N$.
We may assume that $S_i^n \, \cap \, \supp v \neq \emptyset$ 
for some $n>N$ and $i$.
Then we can use the almost increasingness of $r\mapsto r^{d/p}\phi(x,r)$ 
and the condition \eqref{NC} on $\phi$
to get $1/\{\phi(S_i^n)|S_i^n|^{1/p}\} \le C/\{\phi(S_1^N)|S_1^N|^{1/p}\}$ 
for $n>N$,
since $S_i^n \cap S_1^N \neq \emptyset$.
Hence, the last sum is dominated by a constant multiple of
$\|v\|_{L^p}/\{\phi(S_1^N)|S_1^N|^{1/p}\}$
which is small provided $N$ is large, by \eqref{lim2}.
Next, consider the second sum.
If $N$ is large 
we have $1/\phi(S_i^n)<\epsilon/(c\|v\|_{L^\infty})$ for $n<-N$,
from \eqref{lim1}.
Then
\begin{align*}
 \Big|\int_{S_i^n} b_i^{n,k}(x)v(x)\,dx\Big| 
 &\le
 \|b_i^n\|_{L^q} \left(\int_{S_i^n}|v(x)|^p\,dx\right)^{1/p} \\
 &\le
 \frac{1}{\phi(S_i^n)|S_i^n|^{1/p}} \times \|v\|_{L^\infty} |S_i^n|^{1/p} \\
 &\le 
 \frac{\|v\|_{L^\infty}}{\phi(S_i^n)}
 <\epsilon/c.
\end{align*}
Thus, the second sum is dominated by
\begin{equation*}
 \sum_{n<-N}\sum_i\lambda_i^{n,k}\,
 \Big|\int_{S_i^n} b_i^{n,k}(x)v(x)\,dx\Big| 
 \le
 \sum_{n<-N}\sum_i\lambda_i^{n,k}\,(\epsilon/c)\le \epsilon. 
\end{equation*}
The first sum involves only a finite number of terms,
independently of $k$ 
(here we use the fact that the compact support of $v$
can only meet a finite number of these balls
since their radii range only from $3^{-N}\sqrt{d}$ to $3^N\sqrt{d}$).
Thus, the first sum tends to $\int(\sum_{-N\le n\le N}\sum_i\lambda_i^nb_i^n)v$
which, for the above reasons, is close to $\int fv$ if $N$  is large.
This establishes Lemma~\ref{lem:subseq}.
\end{proof}
%%%------------------------------------------------------

%%%%==========================================
%\begin{rem}\label{rem:Fatou}
%%%%==========================================
%By Lemma~\ref{lem:subseq} combining with 
%Proposition~\ref{prop:sup in CgC} 
%we also see that the space $B^{[\phi,q]}(\R^d)$ has the Fatou property. 
%This result is a generalization of 
%\cite[Theorem~1.2]{Sawano-Tanaka2015}. 
%A related result is in 
%\cite{Mastyło-Sawano-Tanaka2018}. 
%\end{rem}

%--------------------------------------------------------
Now we prove Theorem~\ref{thm:dual}.
%--------------------------------------------------------

\begin{proof}[\sc Proof of Theorem~\ref{thm:dual}]
Let $X=\overline{\Cic(\R^d)}^{L_{p,\phi}(\R^d)}$ 
and $T=B^{[\phi,q]}(\R^d)$.
Then $T$ is a subspace of $X^*$, since $B^{[\phi,q]}(\R^d)\subset \big(L_{p,\phi}(\R^d)\big)^*$.
If $v\in X$ satisfies $\langle f,v \rangle=0$ 
for all $f\in T=B^{[\phi,q]}(\R^d)$,
then $v=0$.
Thus, $T$ is total.
By Theorem~\ref{thm:total dense} it follows that
$T$ is dense in $X^*$ in the weak* topology.
Since $X$ is a separable normed vector space by Proposition~\ref{prop:separable},
any bounded closed ball in $X^*$ is metrizable, and thus
this is equivalent to the statement: 
if $x^*\in X^*$, then we can find a sequence $\{f_k\}$ in $T$ 
such that $\langle f_k,v \rangle \to \langle x^*,v \rangle$ for all $v\in X$.
That is, the sequence $\langle f_k,v \rangle$ is bounded for each $v\in X$.
By the Banach-Steinhaus theorem we can conclude that 
the sequence of norms $\|f_k\|_{X^*}$ is bounded,
which means that
the sequence of norms $\|f_k\|_{B^{[\phi,q]}}$ is bounded by Proposition~\ref{prop:sup in CgC}.
We can now invoke Lemma~\ref{lem:subseq} 
to obtain a subsequence $\{f_{k_j}\}$ 
and an $f\in B^{[\phi,q]}(\R^d)$ such that
\begin{equation*}
 \langle x^*,v \rangle 
 =
 \lim_{j\to\infty}\langle f_{k_j},v \rangle
 =
 \lim_{j\to\infty}\int_{\R^d} f_{k_j}(x) v(x)\,dx 
 =
 \int_{\R^d} f(x) v(x)\,dx 
\end{equation*}
for all $v\in\Cic(\R^d)$.
From this it follows that the linear functional $x^*$ is represented
by $f$ and we obtain the conclusion.
\end{proof}
%%%===================================================

%%%%%%%%%%%%%%%%%%%%%%%%%%%%%%%%%%%%%%%%%%%%%%%%%%%%%%%%%%%%%%%%%%%%%%
%%%%%%%%%%%%%%%%%%%%%%%%%%%%%%%%%%%%%%%%%%%%%%%%%%%%%%%%%%%%%%%%%%%%%%
\section{Calder\'{o}n--Zygmund operators}\label{s:CZO}
%%%%%%%%%%%%%%%%%%%%%%%%%%%%%%%%%%%%%%%%%%%%%%%%%%%%%%%%%%%%%%%%%%%%%%
%%%%%%%%%%%%%%%%%%%%%%%%%%%%%%%%%%%%%%%%%%%%%%%%%%%%%%%%%%%%%%%%%%%%%%

%%%----------------------------------------------------------------
In this section we prove the boundedness of 
Calder\'{o}n--Zygmund operators on \sloppy $\overline{\Cic(\R^d)}^{L_{p,\phi}(\R^d)}$. 
We first state the definition of  Calder\'{o}n--Zygmund operators 
following \cite{Yabuta1985}. 
Let $\Omega$ be the set of all nonnegative nondecreasing functions $\omega$ 
on $(0,\infty)$ such that $\int_0^1\frac{\omega(t)}{t}\,dt<\infty$. 
%%%------------------------------------------------------------

%%%==============================
\begin{defn}[standard kernel]\label{defn:SK}
%%%==============================
Let $\omega\in\Omega$. 
A continuous function $K(x,y)$ on $\R^d\times\R^d\backslash\{(x,x)\in\R^{2d}\}$
is said to be a standard kernel of type $\omega$ 
if the following conditions are satisfied; 
\begin{gather} 
  |K(x,y)|\le \frac{C}{|x-y|^d}
     \quad\text{for}\quad x\not=y,
                                                  \label{SK1}
 \\
  \begin{split}
  |K(x,y)-K(z,y)|+|K(y,x)-K(y,z)| 
  &\le 
  \frac{C}{|x-y|^d}\,
  \omega\left(\frac{|x-z|}{|x-y|}\right) \\
  &\text{for}\quad |x-y|\ge2|x-z|. 
                                                  \label{SK2}
  \end{split}
\end{gather}
\end{defn}

%%%==============================
\begin{defn}[Calder\'{o}n--Zygmund operator]\label{defn:CZO}
%%%==============================
Let $\omega\in\Omega$. A linear operator $T$ from $\cS(\R^d)$ to $\cS'(\R^d)$ 
is said to be a Calder\'{o}n--Zygmund operator of type $\omega$, 
if $T$ is bounded on $L^2(\R^d)$ 
and there exists a standard kernel $K$ of type $\omega$ such that, 
for $f\in\Cic(\R^d)$, 
\begin{equation*}
  Tf(x)=\int_{\R^d} K(x,y)f(y)\,dy, \quad x\notin\supp f. 
\end{equation*}
\end{defn}
%%%-----------------------------------------------------------

It is known by \cite[Theorem~2.4]{Yabuta1985} 
that any Calder\'{o}n--Zygmund operator of type $\omega\in\Omega$ 
is bounded on $L^p(\R^d)$ for $p\in(1,\infty)$. 
This boundedness was extended to generalized Morrey spaces $L_{p,\phi}(\R^d)$ 
with variable growth function $\phi$ as the following theorem:

%The following theorem is known:
%==================
\begin{thm}[{\cite{Nakai1994MathNachr}}]\label{thm:CZ1}
%==================
Let $T$ be a Calder\'{o}n--Zygmund operator of type $\omega\in\Omega$, 
$p\in(1,\infty)$ and $\phi$ satisfy \eqref{DC}.
Assume that 
there exists a positive constant $C$ such that,
for all $x\in \R^d$ and $r>0$,
\begin{equation}\label{CZ-C}
     \int_r^{\infty}\frac{\phi(x,t)}{t}\,dt 
     \le C \phi(x,r).
\end{equation}
Then 
$T$ is bounded on $L_{p,\phi}(\R^d)$.
\end{thm}

%%%====================
\begin{rem}\label{rem:T}
%%%====================
In the above theorem, we define $Tf$ on each ball $B$ by 
\begin{equation*}
Tf(x)=T(f\chi_{2B})(x)+\int_{\R^d\backslash2B}K(x,y)f(y)\,dy,\quad x\in B. 
\end{equation*}
Then the first term in the right hand side is well defined, 
since $f\chi_{2B}\in L^p(\R^d)$, 
and the integral of the second term converges absolutely. 
Moreover, $Tf(x)$ is independent of the choice of the ball containing $x$. 
\end{rem}

In this section we have the following theorem:
%==================
\begin{thm}\label{thm:CZ2}
%==================
Let $T$ be a Calder\'{o}n--Zygmund operator of type $\omega\in\Omega$.
Let $p\in(1,\infty)$, 
and let $\phi$ be in $\cGdec_p$ 
and satisfy \eqref{NC} and \eqref{CZ-C}. 
Assume that 
\begin{equation}\label{phi x 1}
 \inf_{x\in\R^d}\phi(x,1)>0. 
\end{equation}
Then 
$T$ is bounded 
on $\overline{\Cic(\R^d)}^{L_{p,\phi}(\R^d)}$. 
\end{thm}

%%%------------------------------------------------------------
In a similar way to 
\cite{Yamaguchi2023JMSJ,Yamaguchi-Nakai2021MJIU,Yamaguchi-Nakai-Shimomura2023AMS}
we can apply Theorem~\ref{thm:CZ2} to the dual and bidual operators of $T$.
In general, 
if a linear operator $T$ is bounded from a normed linear space $X$ to a normed linear space $Y$, 
then the dual operator $T^*$ is bounded from $Y^*$ to $X^*$, 
where $X^*$ and $Y^*$ are the dual spaces of $X$ and $Y$, respectively, 
see \cite[Theorem~2'~(p.~195)]{Yosida1980}. 
This idea was used by 
\cite{Rosenthal-Triebel2014,Rosenthal-Triebel2015,Sawano-El-Shabrawy2018} 
for Morrey spaces.
%%%------------------------------------------------------------

%%%===========================
\begin{thm} \label{thm:dual op}
%%%===========================
Let $p\in(1,\infty)$, 
and let $\phi$ be in $\cGdec_p$ and satisfy \eqref{NC}, \eqref{CZ-C}
and \eqref{phi x 1}. 
Let the kernel $K$ satisfy \eqref{SK1} and \eqref{SK2}, and let $K^t(x,y)=K(y,x)$.
Assume that $T$ and $T^t$ are Calder\'{o}n--Zygmund operators 
with kernel $K$ and $K^t$ of type $\omega\in\Omega$,
respectively.
\begin{enumerate}
\item 
The dual operator $T^*$ of $T$ coincides with $T^t$ on $B^{[\phi,p']}(\R^d)$.
Consequently, 
Calder\'{o}n--Zygmund operators of type $\omega\in\Omega$
are bounded on $B^{[\phi,p']}(\R^d)$.
\item 
The bidual operator $T^{**}$ of $T$ coincides with $T$
on $L_{p,\phi}(\R^d)$.
Consequently, 
$T$ is a bounded linear operator on $L_{p,\phi}(\R^d)$. 
\end{enumerate}
\end{thm}

%The boundedness of the Calder\'{o}n--Zygmund operators on $L_{p,\phi}(\R^d)$ 
%and its predual,
%see \cite{Nakai1994MathNachr}.% and \cite{Nakai2017SCM}, respectively.
%See also \cite{Chen-Jia-Yang2022} for Campanato-type spaces based on so-called ball Banach function spaces.

%%%---------------------------------------------------------

%%%======================
\begin{rem}\label{rem:CZ-C}
%%%======================
Let $\phi$ satisfy the doubling condition. 
If $\phi$ satisfies \eqref{CZ-C}, 
then $\phi$ is almost decreasing, 
$\phi(x,r)\to0\;\;(r\to\infty)$ and $\phi(x,r)\to\infty\;\;(r\to+0)$, 
see \cite[Lemma~6]{Janson1976}. 
%Indeed, we have, for all $x\in\R^d$, 
%\begin{align*}
%\phi(x,r)
%\ls
%\int_r^{2r}\frac{\phi(x,t)}{t}\,dt
%&\le
%\int_r^\infty\frac{\phi(x,t)}{t}\,dt
%\ls
%\phi(x,r), &\quad
%&\text{for $r>0$}, \\
%-\log r\times\phi(x,1) 
%\ls
%\int_r^{1}\frac{\phi(x,t)}{t}\,dt 
%&\le
%\int_r^\infty\frac{\phi(x,t)}{t}\,dt
%\ls
%\phi(x,r), &\quad
%&\text{for $0<r<1$}. 
%\end{align*} 
The condition \eqref{CZ-C} was used by 
\cite{Nakai1994MathNachr,Nakai2008AMS,Nakai2010RMC,Nakamura-Noi-Sawano2016}, etc. 
\end{rem}
%%%---------------------------------------------------------

%%%=====================
\begin{lem}[{\cite[Lemma~7.1]{Nakai2008AMS}}]\label{lem:r^epsilon}
%%%=====================
If $\phi$ satisfies \eqref{CZ-C}, 
then there exist $\epsilon>0$ and $C_\epsilon>0$ such that
\begin{equation*}
\int_r^\infty\frac{\phi(x,t)t^\epsilon}{t}\,dt
\le C_\epsilon\phi(x,r)r^\epsilon
\quad \quad \text{for} \quad x\in\R^d, r>0. 
\end{equation*}
\end{lem}

%%%==============================================
\begin{lem}[{\cite[Lemma~7.5]{Yamaguchi-Nakai-Shimomura2023AMS}}]
\label{lem:phi12}
%%%==============================================
Let $\phi$ be in $\cGdec_p$ and satisfy \eqref{NC} and \eqref{CZ-C}. 
Then there exists a positive constant $\epsilon$ such that 
both
$$
\phi_1(x,r)=\phi(x,r)r^{\epsilon}
\quad\text{and}\quad
\phi_2(x,r)=\frac{\phi(x,r)r^{\epsilon}}{(|x|+r)^{\epsilon}}
$$
are also in $\cGdec_p$ and satisfy \eqref{NC} and \eqref{CZ-C}.
\end{lem}
%%%-----------------------------------------------------------

%%%----------------------------------------------------

%%%======================
\begin{proof}[\sc Proof of Theorem~\ref{thm:CZ2}]
%%%======================
By Lemma~\ref{lem:phi12} 
there exists a positive constant $\epsilon$ such that 
both
$\phi_1(x,r)=\phi(x,r)r^{\epsilon}$
and $\phi_2(x,r)={\phi(x,r)r^{\epsilon}}/{(|x|+r)^{\epsilon}}$
are in $\cGdec_p$ and satisfy \eqref{NC} and \eqref{CZ-C}.
From Proposition~\ref{prop:Cic-gM} 
it follows that $\Cic(\R^d)\subset L_{p,\phi_i}(\R^d)$ $(i=1,2)$.
Moreover,
$\phi_i \; (i=1,2)$ satisfy the condition \eqref{DC}. 

Now, let $f\in\Cic(\R^d)$. 
Then $Tf\in L^p(\R^d)$ by the boundedness of $T$ on $L^p(\R^d)$.
Hence, for any $\delta>0$, 
there exists $g\in\Cic(\R^d)$ such that $\|Tf-g\|_{L^p}<\delta$.
Since $f\in L_{p,\phi_i}(\R^d)$,
by Theorem~\ref{thm:CZ1} we have $Tf\in L_{p,\phi_i}(\R^d)$ $(i=1,2)$. 
Since $g\in L_{p,\phi_i}(\R^d)$ $(i=1,2)$ also, we get 
\begin{align*}
 \frac{\M_p(Tf-g,B(x,r))}{\phi(x,r)}
 &\le
 \frac{\phi_1(x,r)}{\phi(x,r)}\|Tf-g\|_{L_{p,\phi_1}} \\
 &=
 r^\epsilon\|Tf-g\|_{L_{p,\phi_1}} \to 0 
 \quad \text{as} \quad r\to+0. 
\end{align*}
Similarly, 
\begin{align*}
 \frac{\M_p(Tf-g,B(x,r))}{\phi(x,r)}
 &\le
 \frac{\phi_2(x,r)}{\phi(x,r)}\|Tf-g\|_{L_{p,\phi_2}} \\
 &=
 \frac{r^{\epsilon}}{(|x|+r)^{\epsilon}}\|Tf-g\|_{L_{p,\phi_2}} \to 0
 \quad \text{as} \quad |x| \to \infty. 
\end{align*}
Finally, using the almost increasingness of $r\mapsto r^{d/p}\phi(x,r)$ and \eqref{phi x 1},
we have that, if $r>1$, then
\begin{align*}
 \frac{\M_p(Tf-g,B(x,r))}{\phi(x,r)}
 &\le
 \frac{\|Tf-g\|_{L^p}}{|B(x,r)|^{1/p}\phi(x,r)} \\
 &\le
 \frac{\delta}{|B(0,1)|^{1/p}r^{d/p}\phi(x,r)}
 \le
 \frac{C\delta}{\phi(x,1)}
 \le
 \frac{C\delta}{\inf_{x\in\R^d}\phi(x,1)},
\end{align*}
which shows
\begin{equation}
 \limsup_{r\to\infty}\sup_{x\in\R^d}
 \frac{\M_p(Tf-g,B(x,r))}{\phi(x,r)}
 \le C\delta,
\end{equation}
where $C$ is the constant dependent only on $d$, $p$ and $\phi$.

Thus, by Theorem~\ref{thm:distance} 
and Remark~\ref{rem:distance}  we have 
\begin{equation*}
 \inf_{h\in\Cic(\R^d)}\|Tf-g-h\|_{L_{p,\phi}}
 =
 \inf_{h\in\Cic(\R^d)}\|Tf-h\|_{L_{p,\phi}}
 \le
 C\delta. 
\end{equation*}
That is, 
$Tf\in\overline{\Cic(\R^d)}^{L_{p,\phi}(\R^d)}$. 
We get the conclusion.
\end{proof}

\begin{proof}[\sc Proof of Theorem~\ref{thm:dual op}]
First, we prove (i).
We note that the kernel $K^t$ also satisfies \eqref{SK1} and \eqref{SK2}.
Let $f\in\Cic(\R^d)\subset L_{p,\phi}(\R^d)$ 
and $g\in L_\comp^{p'}(\R^d) \subset B^{[\phi,p']}(\R^d)$. 
Then 
\begin{equation*}
 \langle f,T^*(g) \rangle
 =
 \langle T(f),g \rangle
 =
 \int_{\R^d}T(f)(x)g(x)\,dx 
 =
 \int_{\R^d}f(x)T^t(g)(x)\,dx, 
\end{equation*}
that is, $T^*=T^t$ on $L_\comp^{p'}(\R^d)$,
which shows $T^*=T^t$ on $B^{[\phi,p']}(\R^d)$, 
since $L_\comp^{p'}(\R^d)$ is dense in $B^{[\phi,p']}(\R^d)$.

Next, we prove (ii). 
Let $g\in L_\comp^{p'}(\R^d)$, $\supp g \subset B$ and $f\in L_{p,\phi}(\R^d)$. 
We write $f=f_1+f_2$, where $f_1=f\chi_{2B}$, $f_2=f(1-\chi_{2B})$. 
Then 
\begin{align*}
 \langle g, T^{**}(f_1) \rangle
 &=
 \langle T^*(g), f_1 \rangle
 =
 \langle T^t(g), f_1 \rangle\\
 &=
 \int_{\R^d} T^t(g)(x)f_1(x)\,dx\\
 &=
 \int_{\R^d} g(x)T(f_1)(x)\,dx, 
\end{align*}
and 
\begin{align*}
 \langle g, T^{**}(f_2) \rangle
 &=
 \langle T^*(g), f_2 \rangle
 =
 \langle T^t(g), f_2 \rangle\\
 &=
 \int_{(2B)^c}\left(\int_B K^t(x,y)g(y)\,dy\right)f(x)\,dx\\
 &=
 \int_Bg(y)\left(\int_{(2B)^c}K^t(x,y)f(x)\,dx\right)\,dy. 
%&=
%\int_B T(f_2)(y)g(y)\,dy. 
\end{align*}
Hence, $T^{**}=T$ on $L_{p,\phi}(\R^d)$ by Remark~\ref{rem:T}. 
\end{proof}

%%%========================================================
%%%========================================================
%%%========================================================

\end{document}